\documentclass[11pt,a4paper]{article}
\usepackage{amsmath,amsfonts,amssymb,latexsym,graphics,epsfig,url}
\usepackage{color}
\usepackage{amsthm}
\usepackage[english]{babel}
\usepackage{graphicx}

\newcommand{\executeiffilenewer}[3]{%
\ifnum\pdfstrcmp{\pdffilemoddate{#1}}%
{\pdffilemoddate{#2}}>0%
{\immediate\write18{#3}}\fi%
}

\newtheorem{theorem}{Theorem}[section]
\newtheorem{lemma}[theorem]{Lemma}

\newtheorem{proposition}[theorem]{Proposition}

\newtheorem{definition}[theorem]{Definition}

\def\Cay{\mbox{\rm Cay}}
\def\Coset{\mbox{\rm Coset}}
\def\Z{\ns Z}

\def\vec0{\mbox{\boldmath $0$}}

\def\G{\Gamma}

\def\Z{\ns{Z}}

\def\G{\Gamma}

\def\Z{\mathbb Z}

\begin{document}

\title{From expanded digraphs to lifts of\\
voltage digraphs and line digraphs}

\author{
C. Dalf\'o$^a$, M.A. Fiol$^b$, M. Miller$^c$, J. Ryan$^d$\\
$^{a,b}${\small Departament de Matem\`atiques} \\
{\small Universitat Polit\`ecnica de Catalunya,} \\
{\small Barcelona, Catalonia} \\
{\small {\tt{\{cristina.dalfo,miguel.angel.fiol\}@upc.edu}}} \\
$^{b}${\small Barcelona Graduate School of Mathematics,} \\
{\small  Barcelona, Catalonia} \\
$^c${\small School of Mathematical and Physical Science} \\
{\small The University of Newcastle,}\\
{\small Newcastle, Australia}\\
{\small{\tt {mirka.miller@newcastle.edu.au}}}\\
$^d${\small School of Electrical Engineering and Computer Science}\\
{\small The University of Newcastle,}\\
{\small Newcastle, Australia}\\
{\small{\tt {joe.ryan@newcastle.edu.au}}}
}
\date{}

\maketitle

\begin{abstract}
In this note we present a general approach to construct large digraphs from small ones. These are called expanded digraphs, and, as particular cases, we show the close relationship between lifted digraphs of voltage digraphs and line digraphs, which are two known ways to obtain dense digraphs.
In the same context, we show the equivalence between the vertex-splitting and partial line digraph techniques.
Then, we give a sufficient condition for a lifted digraph of a base line digraph to be again a line digraph.
Some of the results are illustrated with two well-known families of digraphs. Namely, the De Bruijn and Kautz digraphs, where it is shown that both families can be seen as lifts of smaller De Bruijn digraphs with appropriate voltage assignments.
\end{abstract}



\section{Introduction}
In the study of interconnection and communication networks, the theory of digraphs plays a key role because, in many cases, the links between nodes are unidirectional.
In this theory, there are three concepts that have shown to be very fruitful to construct good and efficient networks. Namely, those of quotient digraphs, voltage digraphs and (partial) line digraphs. Roughly speaking, quotient digraphs allow us to obtain a simplified or `condensed' version of a bigger digraph, while the voltage and line digraph techniques do the converse by `expanding' a smaller digraph. From this point of view, it is natural that
the three techniques have close relationships. In this paper we explore some of such interrelations by introducing a general construction that we call expanded digraphs.
These digraphs are obtained from a base graph whose vertices become vertex sets in the new graph, and the adjacencies are defined from a set of mappings.
A special case is obtained when such mappings are defined within a group, so obtaining the lifted graphs of base graphs with assigned voltages (elements of the group) on its arcs. In this context, we show that De Bruijn and Kautz digraphs can be defined as lifted digraphs of smaller De Bruijn digraphs.
Moreover, it is proved that, under some sufficient conditions, the lift of a base graph that is a line digraph is again a line digraph.
In the more general case of nonrestricted maps, we consider the quotient digraphs, and the equivalent constructions of vertex-split digraphs and partial line digraphs.
Here, it turns out that the line digraph and quotient operations commute. Finally, it is proved that the techniques of vertex-splitting and partial line digraph are equivalent, and some consequences are derived.

%
%

\subsection{Background}
Let us first recall some basic terminology and notation concerning digraphs. For the concepts and/or results not presented here, we refer the reader to some of
the basic textbooks on the subject; see, for instance, Chartrand and
Lesniak~\cite{cl96} or Diestel~\cite{d10}.

Through this paper, $\Gamma=(V,E)$ denotes a digraph, with vertex set $V$ and arc set $E$. An arc from vertex $u$ to vertex $v$ is denoted by either $(u,v)$,
$uv$, or $u\rightarrow v$. We allow {\em loops} (that is, arcs from a vertex to itself), and {\em
multiple arcs}.
The set of vertices adjacent to and from $v\in V$ is denoted by $\G^{-}(v)$ and $\G^{+}(v)$,
respectively. Such vertices are referred to as {\em in-neighbors} and {\em out-neighbors} of $v$,
respectively. Moreover, $\delta^-(v)=|\G^{-}(v)|$ and $\delta^+(v)=|\G^{+}(v)|$ are
the \emph{in-degree} and \emph{out-degree} of vertex $v$, and $\G$ is {\em
$d$-regular} when $\delta^+(v)=\delta^-(v)=d$ for any $v\in V$.


\section{Expanded digraphs}
Expanded digraphs are, in fact, a type of compounding that consists of connecting
together several copies of a (di)graph by setting some (directed) edges between any two copies. Let $\Gamma=(V,E)$ be a (base) digraph on $n$ vertices. As said before, we allow loops and multiple arcs.
Assume that each vertex $v\in V$ has assigned a vertex set $U_v$,
and each arc $e=(u,v)\in E$  has assigned a mapping $\phi_{uv}: U_u\rightarrow  U_v$.

\begin{definition}
Let $\Phi=\{\phi_{uv}:(u,v)\in E\}$. The {\em expanded digraph} $\G^{\Phi}$ of $\G$ with respect to $\Phi$ has
vertex set $V(\G^{\Phi})=\{U_v : v\in V\}$, and there is an arc from $x\in U_u$ to $y\in U_v$ whenever $(u,v)\in E$ and $\phi_{uv}(x)=y$.
\end{definition}

Some important particular cases of this construction are obtained when the mappings in $\Phi$ are defined from a group:

\begin{itemize}
\item
{\bf Cayley digraphs}. Let $G$ be a group with generating set $\Delta$ having $\delta$ elements. If $\G$ is a singleton with assigned vertex set $G$,  $\delta$ loops, and each loop $e$ has the mapping  $\phi_e:h\rightarrow hg$, with $g\in \Delta$, then the expanded digraph $\G^{\Phi}$ is the {\em Cayley digraph} $\Cay (G,\Delta)$.
\item
{\bf Coset digraphs}. Let $G$ be a group with generating set $\Delta$ having $\delta$ elements and with subgroup $H$. If $\G$ is a singleton with assigned vertex set $\{Hh: h\in G\}$, $\delta$ loops,  and each loop $e$ has the mapping  $\phi_e:Hh\rightarrow Hhg$, with $g\in \Delta$, then the expanded digraph $\G^{\Phi}$ corresponds to the {\em coset digraph} $\Coset(G,H,\Delta)$.
\end{itemize}
Two natural generalizations of these concepts are the following (as far as we know, the second one is a new proposal):
\begin{itemize}
\item
{\bf Lifted (of voltage) digraphs} or {\bf expanded Cayley digraphs}. Let $G$ be a group with generating set $\Delta$ having $\delta$ elements. If each vertex of $\G$ is assigned to the vertex set $G$, and each arc $e$ has the mapping  $\alpha(e)=\phi_e:h\rightarrow hg$, with $g\in \Delta$, then the expanded digraph $\G^{\Phi}$ is
the so-called {\em lifted digraph} (or simply {\em lift}) $\G^{\alpha}$ (see Section \ref{sec:voltage}).
\item
{\bf Expanded coset digraphs}.  Let $G$ be a group with generating set $\Delta$ having $\delta$ elements and with subgroup $H$. If each vertex of $\G$ is assigned to the vertex set $\{Hh: h\in G\}$, and each arc $e$ has the mapping  $\phi_e:Hh\rightarrow Hhg$, with $g\in \Delta$, then we refer to the corresponding expanded digraph $\G^{\Phi}$ as the {\em expanded coset digraph} $\G^{\alpha}$.
\end{itemize}

\section{Voltage and lifted digraphs}
\label{sec:voltage}

When a group is involved in the setting of the mappings, the symmetry of the obtained constructions  yield digraphs
with large automorphism groups. To our knowledge, one of the first papers where
voltage (undirected) graphs were used for construction of dense graphs was that of Alegre, Fiol and Yebra~\cite{afy86}, but without using the name of `voltage graphs'. This name was coined previously by Gross~\cite{g74}. For more information, see Gross and
Tucker~\cite{gt87}, Baskoro, Brankovi\'{c}, Miller, Plesn\'{\i}k, Ryan and
\v{S}ir\'a\v{n}~\cite{bbmrs97}, and Miller and \v{S}ir\'a\v{n}~\cite{ms}.

Let $\G$ be a digraph with vertex set $V=V(\G)$ and arc set $E=E(\G)$.
Then, given a group $G$ with generating set $\Delta$, a
voltage assignment of $\G$ is a mapping $\alpha:E\rightarrow \Delta$. The lift
$\Gamma^\alpha$ is the digraph with vertex set  $V(\Gamma^\alpha)=V\times G$ and
arc set $E(\Gamma^\alpha)=E\times G$, where there is an arc from vertex $(u,g)$ to
vertex $(v,h)$ if and only if $uv\in E$ and $h=g\alpha(uv)$. Such an arc is denoted by $(uv,g)$.

\subsection{De Bruijn and Kautz digraphs}

Recall that
the {\em De Bruijn digraph}  $B(d,k)$ has vertices $x_1x_2\ldots x_{k}$, where $x_i\in \mathbb{Z}_{d}$ for $i=1,\ldots,k$, and adjacencies
$$
x_1x_2\ldots x_{k}\ \rightarrow\  x_2x_3\ldots x_{k} y,\qquad y\in \mathbb{Z}_{d}.
$$
In contrast with that, the {\em Kautz digraph} $K(d,k)$ has vertices $x_1x_2\ldots x_{k}$, where $x_i\in \mathbb{Z}_{d+1}$, $x_{i}\neq x_{i+1}$ for $i=1,\ldots,k-1$, and adjacencies:
$$
x_1x_2\ldots x_{k}\ \rightarrow\  x_2x_3\ldots x_{k} y,\qquad y\neq x_{k}.
$$
Now, if we consider the mapping
$$
\Phi : x_1x_2\ldots x_{k}\ \mapsto\  x_1;(x_2-x_1)\ldots (x_{k}-x_{k-1}),
$$
two alternative definitions are the following:
The De Bruijn digraph $B(d,k)$ has vertices $\beta_1;\beta_2\beta_3\ldots \beta_{k}$, where $\beta_i\in \mathbb{Z}_{d}$, and adjacencies:
\begin{equation}
\label{alter-DB}
\beta_1;\beta_2\beta_3,\ldots \beta_{k}\ \rightarrow\ \beta_1+\beta_2;\beta_3\ldots \beta_{k}\gamma,\qquad \gamma\in \Z_d.
\end{equation}
Similarly, the Kautz digraph $K(d,k)$ has vertices $\alpha;\beta_2\beta_3\ldots \beta_{k}$, where $\alpha\in \mathbb{Z}_{d+1}$, $\beta_i\in \mathbb{Z}_{d+1}\setminus \{0\}$, and adjacencies:
\begin{equation}
\label{alter-K}
\alpha;\beta_2\beta_3\ldots \beta_{k}\ \rightarrow\ \alpha+\beta_2;\beta_3\ldots \beta_{k}\gamma,\qquad \gamma\neq 0.
\end{equation}
Notice that the mapping
$$
\Phi^*: x_1x_2\ldots x_d\ \mapsto\  \beta_1\beta_2\ldots \beta_{k-1},
 $$
where $\beta_i=x_{i+1}-x_i$, is a homomorphism from the Kautz digraph $K(d,k)$ to the De Bruijn digraph $B(d,k-1)$. Also, note the inverse mapping
\begin{equation}
\label{inverse}
\Phi^{-1}: \alpha;\beta_1\beta_2\ldots \beta_{k-1}\ \mapsto \ \alpha(\alpha+\beta_1)(\alpha+\beta_1+\beta_2)\ldots (\alpha+\beta_1+\cdots+\beta_{k-1}).
\end{equation}
Now we show that both De Bruijn and Kautz digraphs can be seen as lifts of smaller De Bruijn digraphs with appropriate voltage assignments.

\begin{proposition}
\label{lemmaDB}
The equality  $B(d,k+1)=B(d,k)^\alpha$ holds with
\begin{eqnarray*}
\alpha  :    E(B(d,k)) & \rightarrow & \Z_d \\
       x_1\ldots x_k x_{k+1} & \mapsto & x_{1}.
\end{eqnarray*}
\end{proposition}

\begin{proof}
The vertices of the arc $x_1\ldots x_k\ \rightarrow\ x_2\ldots x_{k+1}$ in  $B(d,k)$, with voltage $x_{k+1}$, give rise to the vertex subsets $X=\{x;x_1\ldots x_k:x\in \Z_d\}$ and
$Y=\{y;x_2\ldots x_{k+1}:y\in \Z_d\}$ in $B(d,k)^{\alpha}$. Moreover, the elements of $X,Y\cong \Z_d$ can be written as
$$
x(x+x_1)(x+x_1+x_2)\ldots(x+x_1+\cdots+x_{k}),\qquad x\in \Z_d,
$$
and
$$
y(y+x_2)(y+x_2+x_3)\ldots(y+x_2+\cdots+x_k+x_{k+1}),\qquad y\in \Z_d.
$$
Thus, if we define the mapping $\alpha: X\mapsto Y$ in such a way that $y=x+x_1$, we get the adjacencies
$$
x(x+x_1)(x+x_1+x_2)\ldots(x+x_1+\cdots+x_{k})\ \rightarrow\  (x+x_1)(x+x_1+x_2)\ldots(x+x_1+\cdots+x_{k}+x_{k+1}),
$$
with $x_{k+1}\in \Z_d$, which correspond to those of the De Bruijn digraph $B(d,k+1)$, as claimed.
\end{proof}

As an example, Figure~\ref{voltage-deBruijn-dibuix} shows how to obtain $B(2,3)$ as a lifted digraph of $B(2,2)$.

\begin{figure}[t]
    \begin{center}
  \includegraphics[width=14cm]{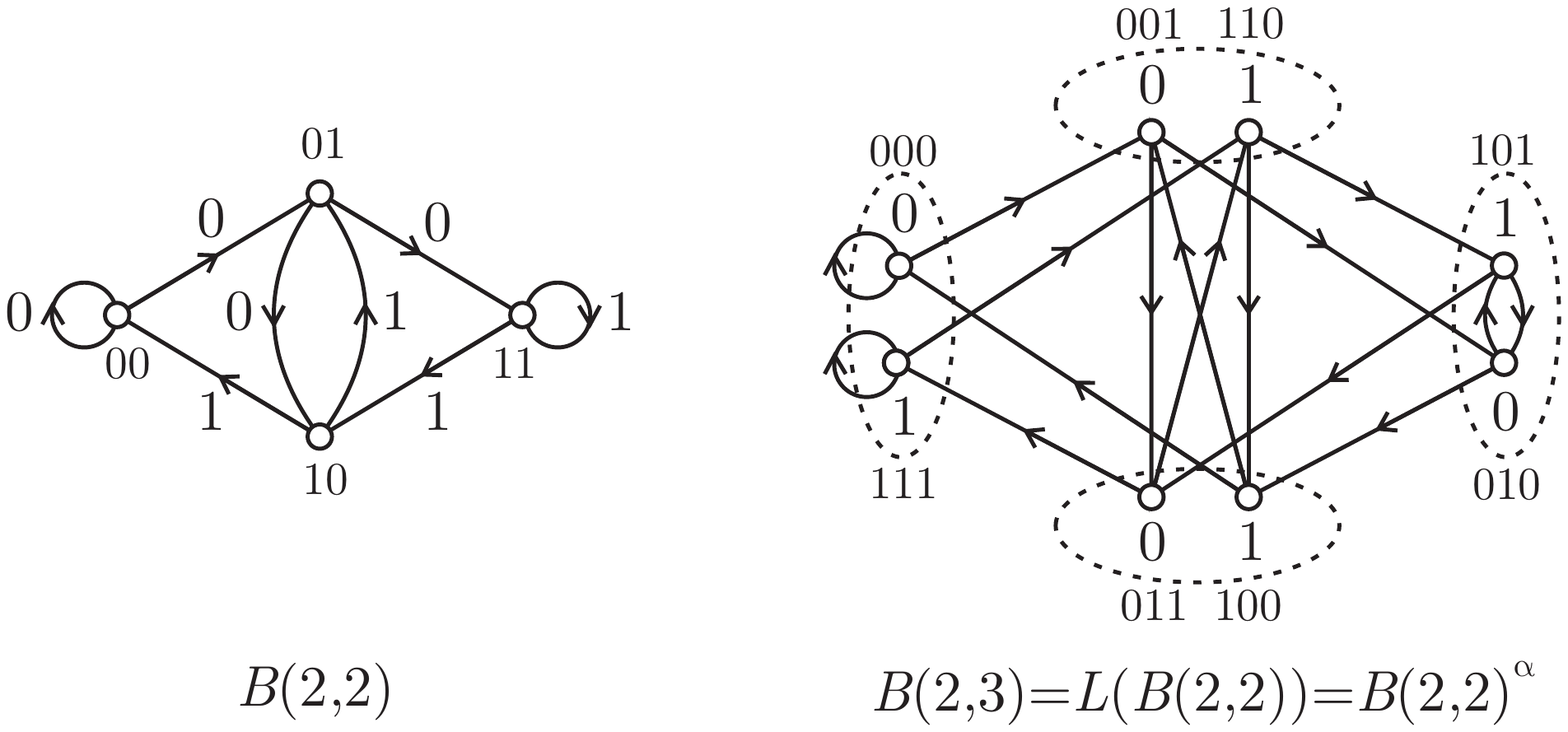}
  \end{center}
  \vskip-14cm
  \caption{$B(2,3)$ as a lifted digraph of the base digraph $B(2,2)$.}
  \label{voltage-deBruijn-dibuix}
\end{figure}

Similarly, the following result shows that Kautz digraphs can be seen as lifted digraphs of De Bruijn digraphs.

\begin{proposition}
The equality $K(d,k+1)=B(d,k)^\beta$ holds with
\begin{eqnarray*}
\beta   :  E(B(d,k)) & \rightarrow & \Z_{d+1} \\
       x_1\ldots x_k x_{k+1} & \mapsto & x_{1},
\end{eqnarray*}
where $x_i\in \Z_{d+1}\setminus \{0\}$.
\end{proposition}

\begin{proof}
The proof is similar to that of Proposition \ref{lemmaDB}. Indeed,
the vertices of the arc $x_1\ldots x_k\ \rightarrow\ x_2\ldots x_{k+1}$ in  $B(d,k)$, with $x_i\in \Z_{d+1}\setminus \{0\}$ and voltage $x_{k+1}$, give rise to the vertex subsets $X=\{x;x_1\ldots x_k:x\in \Z_{d+1}\}$ and
$Y=\{y;x_2\ldots x_{k+1}:y\in \Z_{d+1}\}$ in $B(d,k)^{\beta}$. Moreover, the elements of $X,Y\cong \Z_{d+1}$ can be written as
$$
x(x+x_1)(x+x_1+x_2)\ldots(x+x_1+\cdots+x_{k}),\qquad x\in \Z_{d+1},
$$
and
$$
y(y+x_2)(y+x_2+x_3)\ldots(y+x_2+\cdots+x_k+x_{k+1}),\qquad y\in \Z_{d+1}.
$$
Thus, if we define the mapping $\beta: X\mapsto Y$ in such a way that $y=x+x_1$, we get the adjacencies
$$
x(x+x_1)(x+x_1+x_2)\ldots(x+x_1+\cdots+x_{k})\ \rightarrow\  (x+x_1)(x+x_1+x_2)\ldots(x+x_1+\cdots+x_{k}+x_{k+1}),
$$
with $x_{k+1}\in \Z_{d+1}$, which correspond to those of the Kautz digraph $K(d,k+1)$.
This completes the proof.
\end{proof}

By way of example, in Figure~\ref{K(2,3)voltage-dibuix} we show how $K(2,3)$ can be seen as a lifted digraph of $B(2,2)$.

\begin{figure}[t]
    \begin{center}
  \includegraphics[width=14cm]{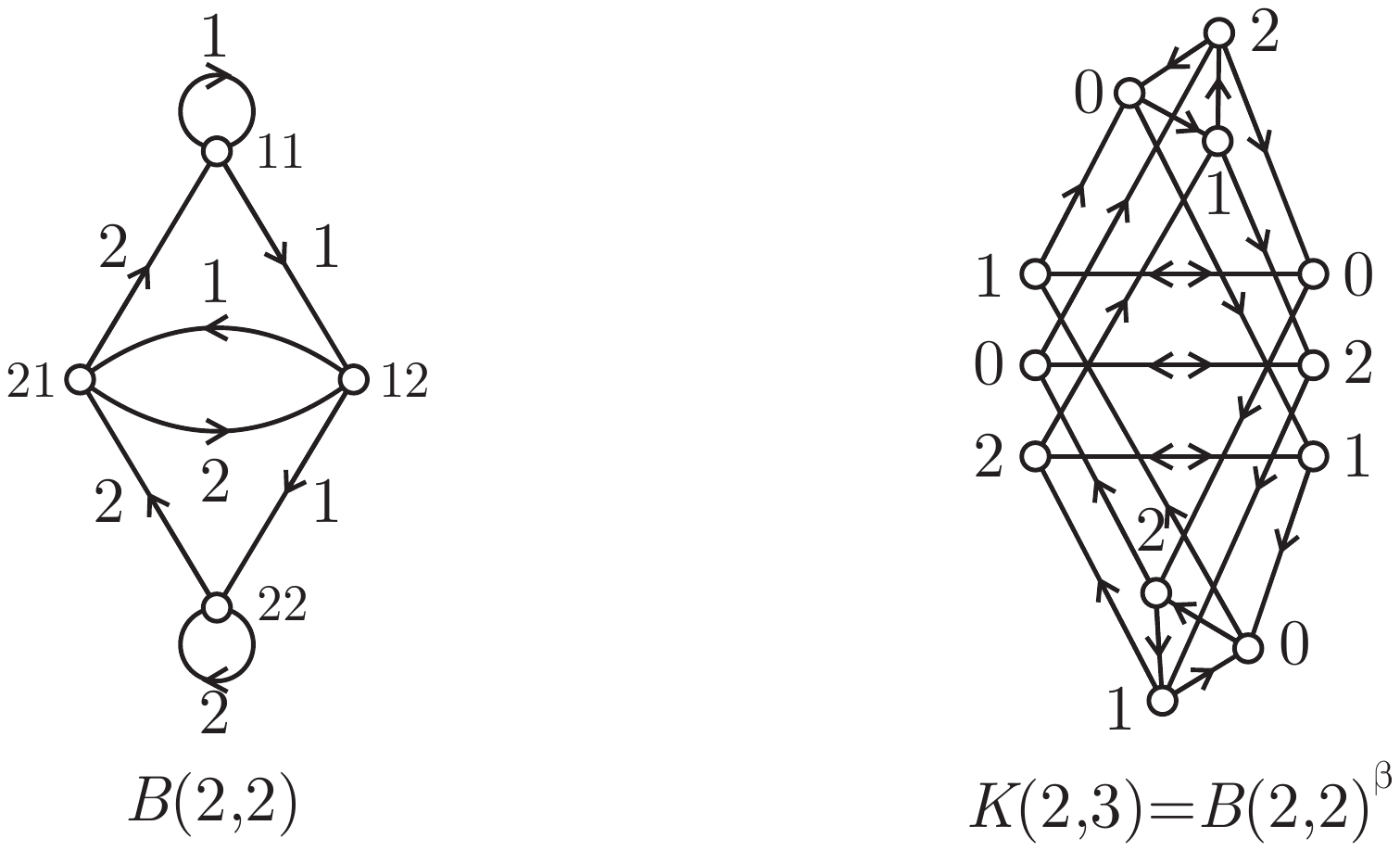}
  \end{center}
  \vskip-13.5cm
  \caption{$K(2,3)$ as a lifted digraph of the base digraph $B(2,2)$.}
  \label{K(2,3)voltage-dibuix}
\end{figure}

\section{Quotient digraphs and line digraphs}
\label{sec:reg-part}

In the more general context of nonrestricted maps, we now consider the quotient digraphs, and the equivalent constructions of vertex-split digraphs and partial line digraphs.

\subsection{Regular partitions and quotient digraphs}
Let $\Gamma=(V,E)$ be a digraph with $n$ vertices. A partition $\pi$ of its
vertex set $V=U_1\cup U_2 \cup\cdots\cup U_m$, for $m\le n$, is called {\em regular} if the number $c_{ij}$ of
arcs from a vertex $u\in U_i$ to vertices in $U_j$ only depends on $i$ and $j$. The numbers $c_{ij}$ are usually called {\em intersection parameters} of the partition.
The {\em quotient digraph} of $\G$ with respect to $\pi$, denoted by $\pi(\G)$, has vertices the subsets $U_i$, for $i=1,\ldots, m$, and $c_{ij}$ parallel arcs from vertex $U_i$ to vertex $U_j$.

\subsection{Line digraphs}
In the {\em line digraph} $L(\G)$ of a digraph $\G$, each vertex
represents an arc of $\G$, that is,
$V(L(\G))=\{uv: (u,v)\in E(G)\}$, and a vertex $uv$ is adjacent to a vertex $wz$ when
the arc $(u,v)$ is adjacent to the arc $(w,z)$: $u\rightarrow v(=w)\rightarrow z$.
Line digraphs have shown to be very interesting structures in the study of dense digraphs (that is, digraphs with a large number of vertices for given degree and diameter). Moreover, it is know that the iteration of the line digraph technique yields digraphs with maximum connectivity. For more details, see, for instance the papers by Fiol, Yebra, and Alegre~\cite{FiYeAl83,FiYeAl84}, and F\`abrega and Fiol~\cite{ff89}. Furthermore, by the Heuchenne's condition~\cite{He64}, a digraph $\G$ is a line digraph if and only if, for every pair of vertices $u$ and $v$, either $\G^+(u)=\G^+(v)$ or $\G^+(u)\cap\G^+(v)=\emptyset$.

\subsection{Regular partitions versus line digraphs}

The following result shows that the quotient and line digraph operations commute.

\begin{proposition}
Every regular partition $\pi$ of a digraph $\G$ induces a regular partition $\pi'$ in its line digraph $L(\G)$ and
\begin{equation}
\label{Lpi=piL}
L(\pi(\G))\cong\pi'(L(G)).
\end{equation}
\end{proposition}
\begin{proof}
Let $\pi=\{U_i : 1\le i\le m\}$ be a regular partition of $\G$ with intersection parameters $c_{ij}$. Then, consider the induced partition of its arcs (or vertices of $L(\G)$) $\pi'=\{U_{ij}:1\le i,j\le m\}$ with sets $U_{ij}=\{u_iu_j : u_i\in U_i, u_j\in U_j\}$. This partition is also regular since, from the definition of line digraph, the number of arcs from a vertex $u_iu_j\in U_{ij}$ to any vertex $u_ku_h\in U_{kh}$ is $c_{jk}$. To prove the digraph isomorphism \eqref{Lpi=piL}, assume that $L(\pi(\G))$ has the arc $U_iU_j\rightarrow U_jU_k$. This means that in $\G$ there is the path $u_i\rightarrow u_j\rightarrow u_k$, where $u_i\in U_i$, $u_j\in U_j$ and $u_k\in U_k$. But $u_iu_j\in U_{ij}$ and $u_ju_k\in U_{jk}$ so that,
in $\pi'(L(G))$, there is the arc $U_{ij}\rightarrow U_{jk}$ and this concludes the proof.
\end{proof}

In Figure~\ref{quotient-line-dibuix} we show an example of this `commutative property'.

\begin{figure}[t]
    \begin{center}
  \includegraphics[width=14cm]{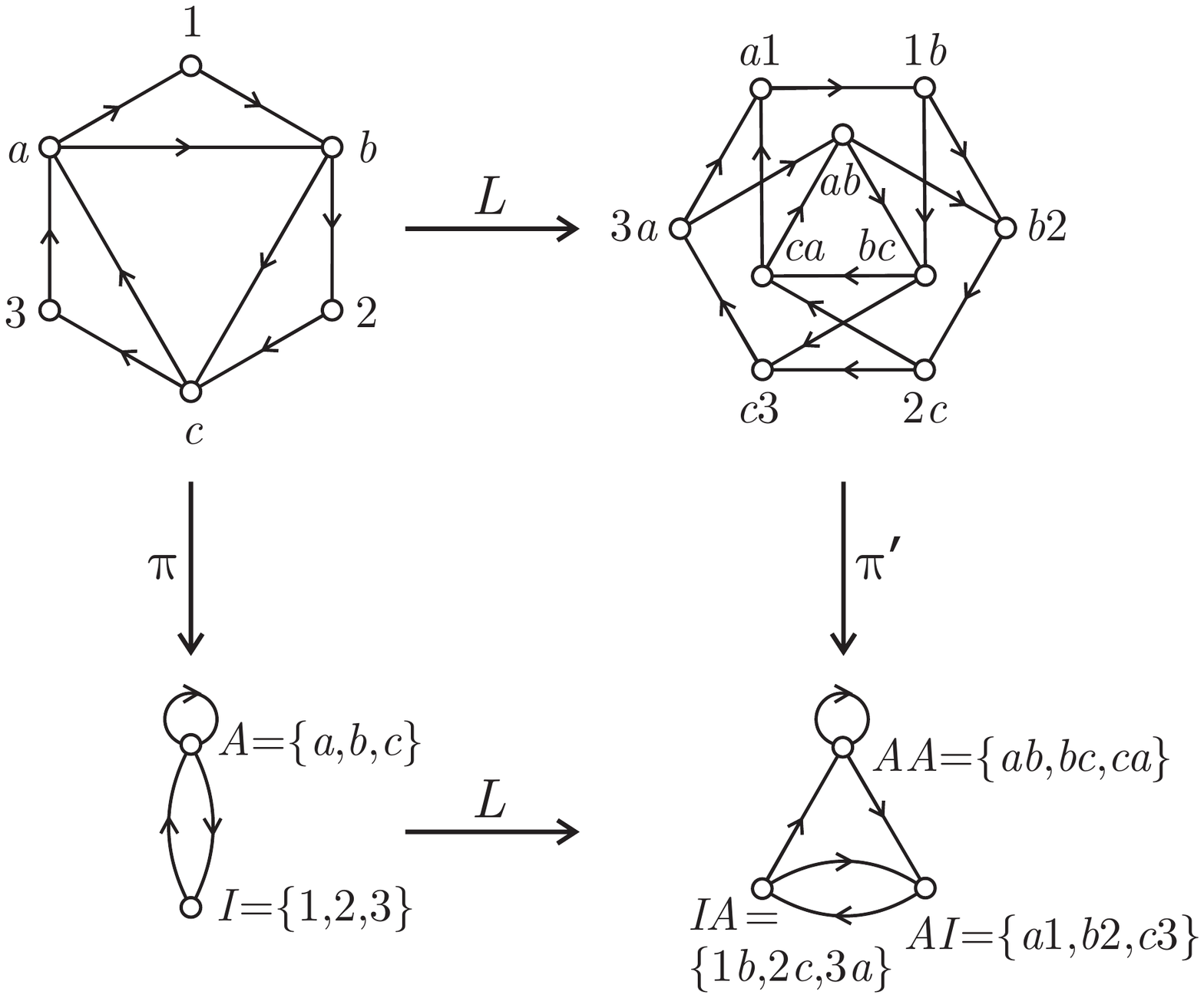}
  \end{center}
  \vskip-9.5cm
  \caption{The line and quotient digraph operations commute.}
  \label{quotient-line-dibuix}
\end{figure}


\subsection{From line digraphs to line digraphs}
Now we show a sufficient condition for being a line digraph be kept when applying the voltage digraph technique.

\begin{proposition}
\label{prop:lift-is-a-LD}
Let $\G=(V,E)$ be a digraph endowed with a voltage assignment $\alpha$. If $\G$ is a line digraph, and for every pair of  vertices $u$ and $v$ with common out-neighbor sets
$\G^+(u)=\G^+(v)=\{x_1,\ldots,x_{\delta}\}$, we have
\begin{equation}
\label{line-condition}
\alpha(ux_i)\alpha(vx_i)^{-1}=\alpha(ux_j)\alpha(vx_j)^{-1},\qquad
i,j=1,\ldots,\delta,
\end{equation}
then, the lifted digraph $\G^{\alpha}$ is again a line digraph.
\end{proposition}

\begin{proof}
It suffices to prove that $\G^{\alpha}$ satisfies Heuchenne's condition.
With this aim, let $ux_i,ux_j\in E$, so that in $\G^{\alpha}$ the vertex $(u,g)$
is adjacent to the vertices $(x_i,g\alpha(ux_i))$ and $(x_j,g\alpha(ux_j))$. Now,
if there is a vertex $v\in V$ such that $vx_i\in E$, we also have $vx_j\in E$ (because
$\G$ is a line digraph). But the former implies that vertex $(v,h)$, where
$h=g\alpha(ux_i)\alpha(vx_i)^{-1}$ is adjacent to vertex
$$
(x_i,h\alpha(vx_i))=(x_i,g\alpha(ux_i)).
$$
Moreover, from (\ref{line-condition}) interchanging $i$ and $j$, we get that
$\alpha(vx_j)=\alpha(vx_i)\alpha(ux_i)^{-1}\alpha(ux_j)$.

Thus, the vertex $(v,h)$
is also adjacent to the vertex
$$
(x_j,h\alpha(vx_j))=(x_j,g\alpha(ux_j)).
$$
Consequently, the vertices $(u,g)$ and $(v,h)$ satisfy Heuchenne's condition and
$\G^{\alpha}$ is a line digraph.
\end{proof}

\subsection{Vertex-splitting}

Let us now consider what we call {\em the vertex-splitting method} to ``blow up'' a digraph.
Given a digraph $\G=(V,E)$ on $n$ vertices and $m$ arcs, the vertex-split digraph $S_{\mu}(\G)$, where $n\le \mu \le m$ is constructed as follows. Every vertex $v\in V$ is split into $\iota(v)$ vertices $v_1,\ldots, v_{\iota(v)}$, where $\iota(v)\le \delta^-(v)$. Thus the order of
$S_{\mu}(\G)$ is
$$
\mu=\sum_{v\in V}\iota(v)\le \sum_{v\in V}\delta^-(v),
$$
satisfying $n\le \mu \le m$.
Moreover, for each arc $vw\in E$, we choose any vertex, say $w_{j}$, with $1\le j\le \iota(w)$,  and set the arcs $v_iw_j$ for every $i=1,\ldots,\iota(v)$. We proceed in this way with all the arcs of $\G$, with the condition that, in the end, all vertices of $S_{\mu}(\G)$ must have nonzero indegree. (Notice that this is always possible, as $\iota(u)\le \delta^-(u)$ for every $u\in V$.)

\subsection{Partial line digraphs}

The above method is shown to be equivalent to the partial line digraph technique, that was proposed by Fiol and Llad\'o in \cite{fl92}, which is as follows.
Given the digraph $\G(V,E)$ as above, let $E'\subseteq E$ a subset of $\mu$ arcs satisfying $\{v: uv\in E'\}=V$, so that $n\le\mu\le m$. Then in the {\em partial line digraph} of $\G$, denoted by $L_{\mu}(G)$, each vertex represents an arc of $E'$, and a vertex $uv$ is adjacent to the vertices $v'w$ for each $w\in\G^+(v)$, where
\begin{enumerate}
\item
$v'=v$, if $vw\in E'$,
\item
$v'$ is any other vertex of $\G$ such that $v'w\in E'$, otherwise.
\end{enumerate}

\begin{lemma}
By appropriately chosen the above vertices $v'$ in the construction of the vertex-split and partial line digraphs of the digraph $\G$, we have
the isomorphism
\begin{equation}
\label{S=L}
S_{\mu}(\G)\cong L_{\mu}(\G).
\end{equation}
\end{lemma}

\begin{proof}
Let the partial line digraph $L_{\mu}(\G)$ be constructed from the arc set $E'$. Then, in constructing $S_{\mu}(\G)$, every vertex $v$ of $\G$ is split into the vertices $v_1,\ldots,v_{\iota(v)}$ if and only if $v_jv\in E'$ for every $j=1,\ldots,\iota(\G)$.
Now, for every $v\in V$, assume that $uv\in E'$. Then, in $S_{\mu}(\G)$ we have $v_i=u$ for some $j=1,\ldots,\iota(v)$.
Assuming that $vw\in E$, we have to consider two cases:
\begin{itemize}
\item
If $vw\in E'$, then in $S_{\mu}(\G)$ we choose every vertex $v_i$, for $i=1\ldots,\iota(v)$, to be adjacent to the vertex $w_j=v$.
\item
Otherwise, in $L_{\mu}(\G)$ we choose the vertex $uv$ to be adjacent to the vertex $v'w$, where $v'=w_j$ (for the chosen vertex $w_j$ in $S_{\mu}(\G)$).
\end{itemize}
Then, it is clear that this gives the claimed isomorphism between both digraphs.
\end{proof}
In particular, when $\iota(u)=\delta^{-}(u)$ for every $u$, we have $S_{m}(\G)\cong L(\G)$.

As an straightforward consequence of the isomorphism \eqref{S=L} we have the following result.
\begin{lemma}
Every partial line digraph $L_{\mu}(\G)$ of a digraph $\G$ can be seen as an expanded digraph of $\G$ with appropriate set $\Phi$ of mappings.
\end{lemma}

Moreover, the diameter and mean distance of the vertex-split digraph $S_{\mu}(\G)$ is only increased by at most one.

\begin{proposition}
Let $\G$ be a digraph different from a cycle, with diameter $D$ and mean distance $\overline{D}$. Then, the diameter $D^*$ and
mean distance $\overline{D}^*$ of the vertex-split digraph $S_{\mu}(\G)$, with $\mu>m$, satisfy:
\begin{eqnarray*}
D^* &= & D+1,\\
\overline{D}^* &< & \overline{D}+1.
\end{eqnarray*}
\end{proposition}
\begin{proof}
This is again a consequence of the isomorphism \eqref{S=L}, together with the result proved in the context of partial line digraphs (see Fiol and Llad\'{o}~\cite{fl92}).
\end{proof}

\noindent{\large \bf Acknowledgments.}
This research is supported by the
{\em Ministerio de Ciencia e Innovaci\'on}, Spain, and the {\em European Regional
Development Fund} under project MTM2014-60127-P, and the {\em Catalan Research
Council} under project 2014SGR1147 (C. D. and M. A. F.).

\end{document}